\documentclass[12pt, reqno]{amsart}
\usepackage{amsfonts,mathrsfs,amssymb,amsmath,amsthm,bm,url,}

\newtheorem{thma}{Lemma}
\newtheorem{thmb}{Theorem}

\makeatletter
\def\imod#1{\allowbreak\mkern10mu({\operator@font mod}\,\,#1)}
\makeatother
\newcommand{\txt}{\textrm}
\newcommand{\xx}{\textbf{x}}
\newcommand{\aaa}{\textbf{a}}

\newcommand{\yy}{\textbf{y}}
\newcommand{\www}{\textbf{w}}
\newcommand{\zz}{\textbf{z}}

\newcommand{\rrr}{\textbf{r}}

\newcommand{\beq}{\begin{equation}}
\newcommand{\eeq}{\end{equation}}
\newcommand{\beqq}{\begin{equation*}}
\newcommand{\eeqq}{\end{equation*}}
\newcommand{\bal}{\begin{align}}
\newcommand{\eali}{\end{align}}
\newcommand{\ball}{\begin{align*}}
\newcommand{\ealii}{\end{align*}}
\newcommand{\eps}{\varepsilon}

\numberwithin{equation}{section}
\numberwithin{equation}{section}
\begin{document}
\title[Weak Approximation for Cubic Hypersurfaces]{Weak Approximation for Cubic Hypersurfaces of Large Dimension}
\author{Mike Swarbrick Jones}
\address{  School of Mathematics, Bristol University, Bristol, BS8  1TW}  

\email{mike.swarbrickjones@bristol.ac.uk}

\begin{abstract} We address the problem of weak approximation for general cubic hypersurfaces defined over number fields, with arbitrary singular locus.  In particular, weak approximation is established for the smooth locus of projective, geometrically integral, non-conical cubic hypersurfaces, of dimension at least 17.  The proof utilises the Hardy-Littlewood circle method, and the fibration method.   
\end{abstract}

\maketitle

\section{Introduction}  Let $k$ be an algebraic number field.  The possible existence and structure of $k$-rational points on hypersurfaces defined over $k$ is a major theme in number theory and arithmetic geometry.  Let $X \subset \mathbb{P}_{k}^{n-1}$ be a variety defined over $k$.  Given a place $\nu$ of $k$, define $k_\nu$ to be the completion with respect to that place.  If $X$ is smooth, recall that \textit{weak approximation} holds for $X$ if $X(k) \neq \emptyset$ and the image of the diagonal embedding $$X(k) \rightarrow \prod_{\nu \in S} X(k_\nu)$$ is dense, for any finite set of places $S$.  Given a possibly singular $X$, we shall consider weak approximation for $X_{smooth}$, the smooth locus of $X$.

We say that $X$ is $k$-rational if there is a $k$-birational morphism \linebreak $\mathbb{P}^{n-1}_k \rightarrow X$.  Weak approximation is a birational invariant on smooth integral varieties, and since weak approximation holds on $\mathbb{P}_k^m$, for any positive integer $m$, it must hold for any smooth $k$-rational variety.  

A classical observation is that a quadric hypersurface $Q$ with a non-singular $k$-point will be $k$-rational, provided it is geometrically integral.  Essentially this is because we can parameterise the surface by lines through the $k$-point.  In this case, the smooth locus of the quadric satisfies weak approximation.  The Hasse-Minkowski theorem implies that $Q_{smooth}(k) \neq \emptyset \iff Q_{smooth}(k_\nu) \neq \emptyset$ for all places $\nu$ of $k$.

For larger degree hypersurfaces, relatively little is known.  The emergence of counterexamples to weak approximation, even when rational points are present, is an indication that the situation is much more complex.  For instance, with $k = \mathbb{Q}$ we have a cubic surface of Swinnerton-Dyer \cite{SwinnyD},

\beq\label{Swinny} x_1(x_2^2 + x_3^2) = (4x_4-7x_1)(x_4^2 - 2 x_1^2).\eeq  Ignoring the subvariety $x_1 = x_4 = 0$, this has two components over the reals: one with $x_4/x_1 \geq 7/4$, which contains infinitely many rational points, the other with $|x_4/x_1|\leq \sqrt{2}$ which contains none.  Clearly then this fails weak approximation.  This counterexample can be accounted for by the \textit{Brauer--Manin Obstruction}.  A conjecture of Colliot-Th\'{e}l\`{e}ne (see \cite{Colliot0}) states that this will be the only such obstruction for rationally connected varieties, such as cubic hypersurfaces of dimension at least 2.

Suppose (once and for all) that $Y \subset \mathbb{P}_k^{n-1}$ is a geometrically integral, non-conical cubic hypersurface, given by the zero locus of a cubic form $C \in k[x_1, \dots , x_n]$.  The Brauer group of $Y_{smooth}$ will be trivial if its dimension is at least 3, and the codimension of the singular locus is at least 4 (see the appendix of Colliot-Th\'{e}l\`{e}ne in \cite{Tim1}).  Thus we expect that these assumptions, together with $Y(k) \neq \emptyset$, imply that weak approximation holds for $Y_{smooth}$.

Let us now consider a few cases where weak approximation for $Y_{smooth}$ is actually known.  Firstly a classical remark.  If $Y_{sing}(k)\neq \emptyset$, where $Y_{sing}$ is the singular locus, then we can paramaterise $Y$ by means of lines through a rational singular point, so $Y$ is $k$-rational.  If $Y$ contains 2 conjugate singular points, and $n \geq 7$, it follows from work of Harari \cite{Harari2}.  If $Y$ contains 3 conjugate singular points and $Y(k_\nu) \neq \emptyset$ for all $\nu$, it is known for $n=4$ (Coray \cite{Coray1}, Coray and Tsfasman \cite{Coray2}) and for $n \geq 6$ (Colliot-Th\'{e}l\`{e}ne and Salberger \cite{Colliot2}), but counterexamples exist for $n=5$ (see \cite[Section 8]{Colliot2}).  If $n \geq 5$, $Y$ is smooth, and contains a $k$-rational line, then it follows from Harari \cite{Harari}.  Finally, a result of Skinner \cite{SK}, shows that if $Y$ is smooth, $n \geq 17$ is sufficient. 

Note that all the results mentioned so far rely fundamentally on the shape of the singular locus of $Y$, either that it is empty, or contains some arithmetic structure.  The aim of this paper is to consider general cubic hypersurfaces $Y$ with arbitrary singular locus, and to obtain a reasonable lower bound on the dimension required to guarantee that $Y_{smooth}$ satisfies weak approximation.  Our main result is the following.

\begin{thmb}\label{main}Suppose we have a geometrically integral, non-conical cubic hypersurface $Y \subset \mathbb{P}_k^{n-1}$ defined over $k$.  If $n \geq 19$, then $Y_{smooth}$ satisfies weak approximation.  \end{thmb}  
In a qualitative sense this result is best possible, in that the geometrically integral and non-conical assumptions cannot be eliminated.  For example, we could consider the union of a line and a quadric which does not have any non-singular rational points, or we could take a cone over the surface $\eqref{Swinny}$.

\textbf{Acknowledgements:}  \linebreak
I am indebted to Professor Colliot-Th\'{e}l\`{e}ne and Professor Leep for numerous remarks on an earlier version of this paper.  These ultimately led to a substantially improved version of Lemma $\ref{main2}$, and hence Theorem $\ref{main}$.  I would also like to thank my supervisor Tim Browning for suggesting the original problem, and for his excellent supervision over the course of the project.  Finally, I am grateful to Dan Loughran for several useful discussions.

\section{Structure of the Proof}  Theorem $\ref{main}$ is related to the result of \cite{SK}, which was obtained using the Hardy-Littlewood circle method.  This is advantageous when the dimension of the singular locus is small.  The circle method can also be an effective tool when the equations involved have large `$h$-invariant'.  This concept was originally introduced by Davenport and Lewis \cite{Dav2}.  

Given a cubic form $C$, define the $h$-invariant of $C$, $h=h_k(C)$ as follows: $h$ is the smallest positive integer such that $C(\xx)$ is expressible identically as \beq\label{hinv} L_1(\xx)Q_1(\xx)+\cdots + L_h(\xx)Q_h(\xx),\eeq where $L_i$, $Q_i$ are linear, quadratic forms respectively, with coefficients in $k$.  Similarly for the cubic hypersurface $Y$, we shall define $h_k(Y)$ to mean the $h$-invariant of the underlying cubic form.  Finally, for a cubic polynomial $f$, we define $h_k(f)$ to be the $h$-invariant of the homogeneous cubic part of the polynomial.  Clearly it is an invariant with respect to non-singular linear transformations on $\xx$ over $k$.  Also note that $h_k(Y) \leq n$, with equality if and only if $Y(k)= \emptyset$.  Furthermore, if $n \geq h_k(Y) + r + 1$, there is a $k$-rational $r$-plane contained in $Y$, given by $L_i = 0$ in $\eqref{hinv}$.

Our strategy is to show weak approximation for three classes of cubics, the union of which contains all of those considered in Theorem $\ref{main}$.  The first class is cubics for which the $h$-invariant is sufficiently large.

\begin{thma}\label{main1}Suppose we have a geometrically integral, non-conical cubic hypersurface $Y$ defined over $k$.  If $h_k(Y) \geq 16$, then $Y_{smooth}$ satisfies weak approximation.\end{thma}

To prove this we take a cue from the concluding remarks of \cite{SK}, and note that to find $k$-rational points that are $\mathfrak{p}$-adically close to a $\mathfrak{p}$-adic point, it is sufficient to find \textit{integer} points that are in specific classes modulo $\mathfrak{p}^t$ for some integer $t$; this is equivalent to finding integral solutions to a cubic polynomial $f$ where the cubic part has the same $h$-invariant as the original cubic form.  For large $h$, this problem is tailor-made for the circle method.  Indeed, we will use a mild generalisation of a previous result of Pleasants \cite{Pleasants} which obtains an asymptotic expression for the number of integral solutions to $f$ in an expanding region, under the assumption that $h_k(f) \geq 16$.

The second class of cubic hypersurfaces we consider are those for which the dimension is somewhat larger than the $h$-invariant. 

\begin{thma}\label{main2}Suppose we have a geometrically integral, non-conical cubic hypersurface $Y$ defined over $k$.  If $n\geq h_k(Y)+5$, then $Y_{smooth}$ satisfies weak approximation.\end{thma}

This is based on the fibration method (see for example \cite{Colliot0} for a more general description), which reduces the question to proving weak approximation for the fibres of a particular map involving $Y$.  Thanks to the assumptions of the Lemma, the fibres in question are quadrics of dimension at least 4.  As noted in the introduction, a quadric $Q$ will satisfy weak approximation if $Q_{smooth}(k_\nu)\neq \emptyset$ for each place $\nu$ of $k$.  A well known theorem of Hasse \cite{Hasse} tells us that this holds if the underlying quadratic form has rank at least 5, and $Q_{smooth}(k_\nu) \neq \emptyset$ for each real place $\nu$.  We must then find conditions on $Y$ under which we can assume that for a generic fibre $Q$, this is the case. This is achieved using an elementary argument.

Combining Lemmas $\ref{main1}$ and $\ref{main2}$, we see that Theorem $\ref{main}$ is true with $n \geq 20$.  To complete the Theorem, we must consider the case $n=19$, $h_k(Y)=15$.  This is disposed of by the following.

\begin{thma}\label{main3}Suppose we have a geometrically integral, non-conical cubic hypersurface $Y$ defined over $k$.  If $11 \leq h_k(Y) \leq n-4$, then $Y_{smooth}$ satisfies weak approximation.\end{thma}  

The proof uses the same ideas as the proof of Lemma $\ref{main2}$, though is a bit more involved.  In private correspondence with the author, Professor Leep has informed us that he believes he can show that the lower bound on $h_k(Y)$ is redundant.

\section{Proof of Lemma 1}\label{previous work}

First we introduce some notation.  Let $k$ be of degree $d$ over $\mathbb{Q}$, and let $\mathfrak{o}$ be the ring of integers of $k$ with $\mathbb{Z}$-basis $\omega_1, \dots , \omega_d$.  Let $\mathfrak{m}$ be an integral ideal of $\mathfrak{o}$, with $\mathbb{Z}$-basis $\tau_1, \dots , \tau_d$.

Define $\sigma_1, \dots \sigma_{d_1}$ to be the distinct real embeddings of $k$, and let $\sigma_{d_1+1}, \dots$ $ \sigma_{d_1 + 2d_2}$ be the distinct complex embeddings, such that $\sigma_{d_1+i}$ is conjugate to $\sigma_{d_1+d_2+i}$.  Put $k_i$ to be the completion of $k$ with respect to the embedding $\sigma_i$ for $i=1, \dots , d_1 + d_2$.

Define $V$ to be the commutative $\mathbb{R}$-algebra $\oplus^{d_1+d_2}_{i=1}k_{i}  \cong k \otimes_\mathbb{Q} \mathbb{R}$ which has dimension $d$. 
For an element $x \in V$, we write $\pi_i(x)$ for its projection onto the $i$th summand, (so $x= \oplus \pi_i(x)$).  There is a canonical embedding of $k$ into $V$ given by $\alpha \rightarrow \oplus\sigma_i(\alpha)$, and we identify $k$ with its image in $V$.  Under this image, $\mathfrak{m}$ forms a lattice in $V$, and $\tau_1, \dots , \tau_d$ form a real basis for $V$.

We define a distance function $| \cdot |_{\tau}$ on $V$ as follows: 
$$|x|_{\tau} = |x_1 \tau_1 + \cdots + x_d \tau_d|_{\tau} = \max\limits_{i} | x_i |.$$  This extends to $V^n$ in the obvious way: if $\textbf{x} = (x^{(1)}, \dots ,\ x^{(n)}) \in V^n,$ then
$$|\textbf{x}|_{\tau} =  \max\limits_{j} | x^{(j)} |_{\tau}.$$  We note that there will be some constant $c$, dependent only on $k$ and the choice of basis $\tau_1, \dots , \tau_d$, such that \beq\label{cdef} |\pi_i(x)| \leq c|x|_{\tau} \eeq for all $x \in V$ and $1 \leq i \leq d_1+d_2$ (since each $\pi_i$ is linear, this is clear). 
 
For any point $\textbf{b} \in V^n$, let $\mathfrak{B}(\textbf{b})$ be the box \begin{equation}\label{box2}\mathfrak{B}(\textbf{b}) = \{  \xx \in V^n : |\xx-\textbf{b}|_{\tau}<\rho/2  \},\end{equation} where $\rho$ will always be a real number $0 <\rho <1 $. 
 
For any set $\mathcal{{A}} \subset V^n$, and positive real number $P$, we define $P \mathcal{{A}}$ to be the set $\{ \xx \in V^n : P^{-1} \xx \in \mathcal{A} \}$.  Given a polynomial $\psi(x_1, \dots , x_n)$ defined over $k$, we shall be interested in the quantity: $$\mathcal{N}_{\psi, \mathcal{{A}}, \mathfrak{m}}( P ) = \# \{ \xx \in P\mathcal{{A}} \cap \mathfrak{m}^n : \psi(\xx)=0 \},$$ and its asymptotic behaviour as $P \rightarrow \infty$.

We can now state the generalisation of the main theorem of \cite{Pleasants} we shall use: 
\begin{thma}\label{Pleasantsmain} (Pleasants) Let $\mathfrak{m}$ be an integral ideal of $\mathfrak{o}$, and let $f(\textup{\textup{\xx}})$ be a cubic polynomial over $k$.  Suppose that $h_k(f) \geq 16,$ and that for every integral ideal $\mathfrak{a}$ of $\mathfrak{o}$, the congruence \begin{equation}\label{pleasantscong} f(\textup{\xx}) \equiv 0 \mod{\mathfrak{a}} \end{equation} has non-singular solutions in $\mathfrak{m}^n$.  Also, let $\boldsymbol{\zeta}_0 = \bigoplus_{i=1}^{d_1+d_2} \boldsymbol{\zeta}_i$ where each $\boldsymbol{\zeta}_i \in \pi_i(V)^n$ is a non-singular solution to $C(\textup{\xx})=0$.  Then there exists a set $\mathfrak{R} \subset V^n$ containing $\boldsymbol{\zeta}_0$, and a real constant $c_{f, \mathfrak{R}, \mathfrak{m}}>0$, such that $$\mathcal{N}_{f, \mathfrak{R}, \mathfrak{m}}( P ) = c_{f, \mathfrak{R}, \mathfrak{m}} P^{(n-3)d}+o(P^{(n-3)d}).$$ \end{thma}

\begin{proof}In the case where $\mathfrak{m}=\mathfrak{o}$, this is equivalent to Lemmas 6.1, 7.1, 7.2 and 7.4 of \cite{Pleasants}, which were proved using the circle method.  However all the arguments go through unchanged to prove the generalisation.  Indeed, if one just changes the words `integral points' to `elements of $\mathfrak{m}$', and `$\omega_1, \dots , \omega_d$' to `$\tau_1, \dots , \tau_d$' in the relevant places, essentially all the arguments work verbatim in the same way.  In terms of the circle method, the commutative algebra $V$ does not behave differently whether $\mathfrak{m}$ is the ring of integers or an arbitrary integral ideal, and the non-trivial algebraic number theory results required (section 4 of \cite{Pleasants}) were not specific to $\mathfrak{o}$. \end{proof}

It is straightforward to show that a suitable $\boldsymbol{\zeta}_0$ exists.  Thus Lemma $\ref{Pleasantsmain}$ shows that any such cubic polynomial has infinitely many solutions in $\mathfrak{m}^n$.

We now prove Lemma $\ref{main1}$.  Recall $Y$ is the hypersurface associated to a rational cubic form $C$, with $h_k(C) \geq 16$.  It is well known that if $Y$ is not a cone, and $C$ has at least 10 variables, the congruences $\eqref{pleasantscong}$ have non-singular solutions in $\mathfrak{o}^n$ for all ideals $\mathfrak{a}$ (see for example \cite{BirchLewis}).   Then it is clear upon taking $\mathfrak{m}=\mathfrak{o}$ in Lemma $\ref{Pleasantsmain}$, that $Y(k) \neq \emptyset$.  Furthermore this implies that $Y_{smooth}(k) \neq \emptyset$, as $Y$ is not a cone (see for example \cite[\txt{Theorem} 1.2]{Kollar}).  Suppose we are given $\eps>0$, any finite set of places $S$, and any set of non-singular points $\{ \xx_{\nu} = (x^{(1)}_\nu, \dots , x^{(n)}_\nu) \in Y_{smooth}(k_\nu): \nu \in S \}.$  To show that weak approximation holds for $Y_{smooth}$, it suffices to show there exists a point $\xx = (x^{(1)}, \dots , x^{(n)})\in Y_{smooth}(k)$ such that $|x^{(i)}-x^{(i)}_\nu|_{\nu} < \eps$ for each $i$, and every $\nu \in S$ (where $| \cdot |_\nu$ is the valuation with respect to $\nu$).  We follow the line of argument of \cite[\txt{Section 5}]{SK}.  

Let $\eps$, $S$, $\{ \xx_\nu \}_{\nu \in S}$ be given.  We may assume that $\txt{ord}_{\nu}(x^{(i)}_\nu) \geq 0$ for every $i$ and every $\nu \in S$.  Write $S=S_\infty \cup S_f$, where $S_\infty$ consists of infinite places and $S_f$ consists of finite places.  Without loss of generality we can assume that $S_\infty$ consists of all the infinite places of $k$, since there are only finitely many of them, and $Y_{smooth}(k_\nu) \supset Y_{smooth}(k) \neq \emptyset$ for all $\nu$.  

We can find $\aaa = (a^{(1)}, \dots, a^{(n)} ) \in \mathfrak{o}^n$ such that $|a^{(i)} - x^{(i)}_\nu|_\nu < \eps/3$ for all $i$ and $\nu \in S_f$ (by the Chinese Remainder Theorem).  Let $$r_\nu = \min_i \txt{ord}_\nu \{ a^{(i)} - x^{(i)}_\nu \},$$ and let $\mathfrak{p}_\nu$ be the prime ideal corresponding to $\nu$.  Put $$\mathfrak{m}= \prod_{\nu \in S_f} \mathfrak{p}_\nu^{r_\nu}.$$  

Consider $f(\xx)=C(\xx + \aaa)$, a cubic polynomial defined over $k$.  Let $t$ be a positive integer.  Choose $D \equiv 1 \pmod{\mathfrak{m}^t}$ to be a positive integer such that $$D > \frac{2c}{\eps},$$ with $c$ as in $\eqref{cdef}$.  For each infinite place $\nu$, let $$\rrr_\nu= D \xx_\nu.$$  Put $$\boldsymbol{\zeta}_0 = \bigoplus_{i=1}^{d_1+d_2} \rrr_{\nu_i} \in V^n,$$ where $\nu_i$ is the infinite place corresponding to the embedding $\sigma_i$.  Note that $\boldsymbol{\zeta}_0$ satisfies the conditions of Lemma $\ref{Pleasantsmain}$.  Take a set $\mathfrak{R}$ as in Lemma $\ref{Pleasantsmain}$, centred at $\boldsymbol{\zeta}_0$.  In \cite{Pleasants}, the region $\mathfrak{R}$ is essentially a box-like shape, and the only extra condition it needs to have is that it is sufficiently small.  Therefore we can take its `diameter' with respect to $|\cdot|_\tau$ to be as small as we like, and we can assume it is contained inside a box of side length $\rho<1$ as in $\eqref{box2}$.  

We consider the congruence conditions $\eqref{pleasantscong}$.  For any finite place $\nu \not\in S_f$, we have $\mathfrak{o}_\nu = \mathfrak{m}_\nu$, and so any point in $Y_{smooth}(k_\nu)$ will give rise to a non-singular solution in $\mathfrak{m}_\nu^n$ of $f(\xx)=0$.  On the other hand, if $\nu \in S_f$, $\xx_\nu-\aaa \in \mathfrak{m}_\nu^n$ is a non-singular solution to $f(\xx)=0$.  Thus the conditions hold for all integral ideals $\mathfrak{a}$ in $k$, by the Chinese remainder theorem. 

Finally, we note that the cubic part of $f$ is just $C$, and  $h_k(C) \geq 16.$   Thus the conditions for Lemma $\ref{Pleasantsmain}$ are met.  Hence for a sufficiently large integer $P \equiv 1 \pmod{\mathfrak{m}^t}$, there exists a point $\yy \in \mathfrak{m}^n \cap  P \mathfrak{R}$ which is a zero of $f$, and thus $\zz = \yy + \aaa$ is a zero of $C$.  Also, since we can have arbitrarily many such points, we can choose one such that $\zz \neq\textbf{0}$.  We now fix our point $\xx$ to be $\frac{\zz}{DP}$.

For $\nu \in S_ \infty$ we have: $$|DP x_\nu^{(i)} -  y^{(i)}|_{\nu} \leq \rho c P < cP,$$ whence \begin{equation*}\begin{aligned}
\big| x_\nu^{(i)} - x^{(i)} \big|_{\nu} &= \big| x_\nu^{(i)} - \frac{z^{(i)}}{DP} \big|_{\nu} \\ &\leq \big|x_\nu^{(i)} - \frac{ y^{(i)}}{DP}\big|_{\nu} + \frac{|a^{(i)}|_{\nu}}{DP} \\
&\leq \frac{c}{D} + \frac{|a^{(i)}|_{\nu}}{DP} \\
&< \eps
\end{aligned}
\end{equation*} for $P$ sufficiently large.

For every $\nu \in S_f$ we have \begin{align*} \big|x_\nu^{(i)} - x^{(i)} \big|_\nu &= \big| x_\nu^{(i)} - \frac{z^{(i)}}{DP} \big|_\nu \\ &\leq |x_\nu^{(i)}-z^{(i)}|_\nu + \frac{|DP-1|_\nu}{|DP|_\nu}|z^{(i)}|_\nu \\
&\leq |x_\nu^{(i)}-z^{(i)}|_\nu + 2^{-t} \\
&= |(z^{(i)}-a^{(i)}) - (x_\nu^{(i)}-a^{(i)})|_\nu + 2^{-t} \\
&\leq \frac{2 \eps}{3} + 2^{-t} < \eps,
\end{align*}
when $DP \equiv 1 \pmod{M^t}$ for sufficiently large $t$.  Finally we note that by taking $\eps$ sufficiently small, we can make $\xx$ be arbitrarily close to a non-singular point on $Y(k_\nu)$ for each $\nu \in S$.  In this way we may clearly assume that $\xx \in Y_{smooth}(k)$.  This proves Lemma $\ref{main1}$.

\section{Proof of Lemma 2}\label{Lemma 2}

Throughout this section we shall suppose that $h = h_k(Y)$, and that $n \geq h + 5$.  In fact, for simplification, we shall only consider the case $n = h+5$, other cases being handled similarly.  After a change of variables if necessary we can express the cubic form $C$ in terms of variables $(\xx, \yy)= (x_1 \dots , x_5 , y_1, \dots, y_h)$ as follows, \begin{align*}Y : C(\xx , \yy) = y_1 Q_1(x_1, \dots ,x_5 , y_1, &\dots , y_h) + \nonumber \\ \cdots + &y_h Q_h(x_1, \dots ,x_{5} , y_1, \dots , y_h)=0, \end{align*} where the $Q_i$ are quadratic forms defined over $k$.  Clearly $Y(k) \neq \emptyset$ since $h < n$, so $Y_{smooth}(k) \neq \emptyset$ as in the proof of Lemma $\ref{main1}$.

Consider the 4 dimensional linear space $L$ given by $y_1 = \cdots  = y_h = 0$.  Take the blow-up $W$ of $Y$ along $L$.  Let $z_1 , \dots, z_h$ be coordinates for $\mathbb{A}_k^{h}$.  Then the variety given by the vanishing of $C$ and $y_i z_j - y_j z_i$ for all $i,j \in \{1, \dots , h \}$, is $L \cup W$.  Our plan is to prove weak approximation for $W_{smooth}$, which by birationality, will prove it for $Y_{smooth}$.  

Let $\pi:W \rightarrow \mathbb{A}_k^{h}$ be the projection $(\xx, \yy , \zz) \mapsto \zz$.  The generic fibres of this map are quadrics in $\mathbb{P}_k^5$.  Abusing our notation slightly, suppose that for $\yy \in \mathbb{A}_k^h \setminus \{ \textbf{0} \}$, the fibre of $W/\mathbb{A}_k^h$ at the point $\yy$ is given by the quadric \beq\label{Qa}X_\yy: Q_\yy(\xx,t) = 0.\eeq  Then we see that \beq\label{Qa2}Q_\yy(\xx,t) t = C(\xx,\yy t) .\eeq 

For an alternative description of the fibres, consider a generic 5 dimensional linear space $L_5$ which contains $L$.  Then this cuts out on $Y$ the union of $L$ and one such quadric $Q$.

\begin{thma}\label{CSSD}  Let $X$ be a smooth geometrically integral variety such that $X(k) \neq \emptyset$ and $X$ satisfies weak approximation.  Let $Z / X$ be a fibration in quadrics such that the generic fibre has underlying rank at least 5.  Then weak approximation holds for any smooth model of $Z.$  \end{thma}

\begin{proof}This is almost exactly \cite[\txt{Proposition} 3.9]{Colliot}, and the proof follows in the same manner.  The only real difference is that our fibres may be singular.  We leave it as an exercise to the reader to convince themselves that this is only a minor issue.  \end{proof}

Now we can reduce Lemma $\ref{main2}$ to the following.

\begin{thma}\label{second2}  Suppose $Y$ is as in Lemma $\ref{main2}$.  For $\textup{\yy} \in \mathbb{A}_k^h$, we define $Q_\yy$ as in $\eqref{Qa}$.  Either $Q_\yy$ is generically of rank at least 5, or $Y_{sing}(k) \neq \emptyset$.  \end{thma}

If $Q_\yy$ is generically of rank at least 5, we apply Lemma $\ref{CSSD}$ with $Z=W$, $X=\mathbb{A}_k^h$ to prove weak approximation for $W_{smooth}$, and hence $Y_{smooth}$.  In the alternative, $Y_{sing}(k) \neq \emptyset$, which is sufficient for weak approximation as noted in the introduction.  

We now establish Lemma $\ref{second2}$.  Suppose that $Y_{sing}(k) = \emptyset$.  We may write \begin{align*}Y : C(\xx, \yy) = \sum_{i=1}^h y_i  Q_i(\xx) + 2 \sum_{j=1}^5 x_j q_j(\yy) + c(\yy), \end{align*} where $Q_i, q_j$ are quadratic forms, and $c$ is a cubic form.  Then by $\eqref{Qa2}$, $Q_\yy$ takes the form \begin{align}\label{Qform}Q_\yy(\xx,t) =  \sum_{i=1}^h y_i Q_i(\xx) + 2 \sum_{j=1}^5 x_j q_j(\yy) t + c(\yy) t^2. \end{align}  We rewrite this in terms of a $6 \times 6$ matrix $A = A(\yy)$, defined by \beq\label{Adef} Q_\yy(\xx,t) = \, ^T(\xx, t)A(\xx, t) \nonumber. \eeq  Also we consider the $5 \times 5$ submatrices $M_i$ given by $Q_i(\xx)= \, ^T \xx M_i \xx .$

By $\eqref{Qform}$, $A$ must take the shape \beq\label{Adef1} A(\yy) =
\left(\begin{array}{cccccc}
 &   &  & \vline & q_1(\yy)   \\
 & \sum_i y_i M_i &  & \vline & \vdots  \\
 &   &  & \vline & q_5(\yy)     \\
\hline 
q_1(\yy) & \cdots & q_5(\yy) & \vline  & c(\yy) &    \\

\end{array}\right).
\eeq

We shall show that generic $\yy$ produces $\txt{rk}[A(\yy)] \geq 5$.
 
Due to its simple structure, we shall focus all our attention on the leading $5 \times 5 $ submatrix $M = M(\yy):=  \sum_i y_i M_i$, and ignore the terms which are quadratic or cubic in $\yy$ (we shall need to consider these more carefully in proving Lemma $\ref{main3}$).  Clearly $\txt{rk}(A) \geq \txt{rk}(M)$, so it suffices to show $\txt{rk}(M) = 5$, i.e. $\det (M) \neq 0$, generically.

Suppose there is some $\xx$ variable, $x_1$ say, so that none the forms $Q_i$ vary with $x_1$.  This is problematic, since then $ \txt{rk}( \sum_{i=1}^h y_i M_i) < 5$ for any choice of $\yy$.  In fact, let's merely assume that there is no $x_1^2$ term in any of the $Q_i$'s.  This implies that all terms in $C(\xx,\yy)$ are at most linear in $x_1$.  But then $Y_{sing}(k)\neq \emptyset$ , since it contains the point  $(1 , 0, \dots , 0)$, which is a contradiction.  Thus we can assume that there is an $x_j^2$ term contained in at least one of the $Q_i$'s, for every $j$, and that this is true under any non-singular $k$-rational change of variables acting on $\xx$.  Recall that a non-singular change of variables $\xx' = B \xx$ corresponds to a congruent transformation for each $M_i$: $M'_i = B^T M_i B$.  Thus we have reduced Lemma $\ref{second2}$ entirely to the following linear algebra result.
 
\begin{thma}\label{s-5}Suppose we have a set of $\rho \times \rho$ symmetric, $k$-rational matrices $M_1 , \dots, M_\mu$ such that the $i$-th diagonal entry is non-zero for at least one of the matrices, for $i \in \{1, \dots , \rho\}.$  Furthermore, this remains true under any $k$-rational non-singular congruent transformation acting on all the matrices.  Then for generic $\yy \in \mathbb{A}_k^\mu$, we have that $\det ( \sum_{j=1}^\mu y_j M_j) \neq 0. $  \end{thma} 

\begin{proof}

A first observation is that equation $\det ( \sum^\mu_{j=1} y_j M_j) =0$ defines an algebraic set, $\mathcal{A}$ say, in $\mathbb{A}_k^\mu$, so the object of interest is then the set $\mathbb{A}_k^\mu \setminus \mathcal{A}$.  This is either empty or Zariski dense and open, hence we only need to show the existence of one $\yy \in \mathbb{A}_k^\mu \setminus \mathcal{A}$.

To do this we shall use induction on $\rho$, the $\rho = 1$ case being trivial.  We suppose that $\rho > 1$, and that $\mu >1$, since if we have $\mu=1$, then $M_1$ must have full rank, and we are done.  

Clearly the matrices $M_i$ cannot all be 0, suppose $M_1 \neq 0$.  Recall from the classical theory of quadratic forms that a symmetric matrix defined over $k$ is congruent over $k$ to a diagonal matrix.  Therefore without loss of generality, we assume that $M_1$ is in diagonal form, $$M_1 =
\left(\begin{array}{cccccc}
0 &  & & &    \\
 &\ddots & & &   \\
 & & 0 & &     \\
 & & &\lambda_1 &   \\ 

 & & &  &\ddots &    \\
 & & &  & &\lambda_r   \\

\end{array}\right),
$$ where $1 \leq  r = \txt{rk}(M_1)$, $\lambda_i \neq 0$, and all other entries are 0.  Now consider the leading $\rho-r \times \rho-r$ leading submatrices of $M_j$ for $j \geq 2$, which we will denote $M'_j$.  These are all symmetric and defined over $k$.  Furthermore, by the assumption of the lemma, for each $i = 1, \dots \rho-r$, we know that at least one $M'_j$ has its $i$th diagonal entry being non-zero, and that this is true under any rational congruent transformation.  Hence by our inductive hypothesis, for a generic $(y_2 , \dots , y_\mu) \in \mathbb{A}_k^{\mu-1},$ $\det ( \sum_{j=2}^\mu y_j M'_j) \neq 0. $ Let $(b_2, \dots , b_\mu)$ be one such generic vector, and put $M' = \sum_{j=2}^\mu b_j M'_j$.  Suppose $B = \max_i |b_i|$, where here $|\cdot|$ just denotes the usual complex absolute value.  Let $\| \cdot \|$ be the max norm for matrices, i.e if $A = \{ a_{ij} \}$, then $\| A \| = \max\{|a_{ij}| \}$.  Then define $N= \max_{j} \| M_j \|$.

Now consider the matrix $M = P M_1 + \sum_{j=2}^\mu b_j M_j$, where $P$ is a positive real number.  By Laplace expansion, $$\det(M) = \lambda_1 \cdots \lambda_r P^r \det (M') + O_{B,N} \left(P^{r-1} \right).$$ Since $\lambda_1 \cdots \lambda_r \det ( M') \neq 0,$ we see that if we take $P$ sufficiently large, the determinant will be non-zero. \end{proof}

Thus the proof of Lemma $\ref{main2}$ is complete.  

\section{Proof of Lemma 3}  Throughout this section we shall suppose that $h = h_k(Y)$, and that $11 \leq h = n-4$.  We can essentially repeat the argument of Section $\ref{Lemma 2}$ verbatim up to the statement of Lemma $\ref{second2}$.  The structure of $Q_\yy(\xx,t)$ is the same here, except now there are only four $\xx$ variables, \begin{align}\label{Qform2}Q_\yy(\xx,t) =  \sum_{i=1}^h y_i Q_i(\xx) + 2 \sum_{j=1}^4 x_j q_j(\yy) t + c(\yy) t^2. \end{align}  Then Lemma $\ref{main3}$ follows from the following Lemma, the proof of which will make up the remainder of the section.

\begin{thma}\label{second3}  Suppose $Y$ is as in Lemma $\ref{main3}$.  For $\textup{\yy} \in \mathbb{A}_k^h$, we define $Q_\yy$ as in $\eqref{Qform2}$.  Either $Q_\yy$ is generically of full rank, or $Y_{sing}(k) \neq \emptyset$.  \end{thma}

We assume that $Y_{sing}(k) \neq \emptyset$.  First we rewrite $\eqref{Qform2}$ in terms of the $5 \times 5$ matrix $A = A(\yy)$, defined by \beq\label{Mdef2} Q_\yy(\xx,t) = \, ^T(\xx, t)A(\xx, t). \eeq  Also we consider the $4 \times 4$ submatrices $M_i$ given by $Q_i(\xx)= \, ^T \xx M_i \xx .$

Thus $A$ must take the shape \beq\label{Adef2} A(\yy) =
\left(\begin{array}{cccccc}
 &   &  & \vline & q_1(\yy)   \\
 & \sum_i y_i M_i &  & \vline & \vdots  \\
 &   &  & \vline & q_4(\yy)     \\
\hline 
q_1(\yy) & \cdots & q_4(\yy) & \vline  & c(\yy) &    \\

\end{array}\right).
\eeq

We must show under the assumption $Y_{sing}(k) \neq \emptyset$, that the determinant $\det [A(\yy)]$ is not identically zero in $\yy$, as this will show that $\det [A(\yy)]\neq 0$ for generic $\yy$.  Consider the Leibniz determinant formula, which expresses the determinant in terms of products of elements of the matrix which are on `diagonals'.  We shall say a diagonal of $A$ is of \textit{type 1}, if it contains the element $c(\yy)$, and \textit{type 2} otherwise.  Our strategy is to try to distinguish between the two types of diagonal.

Say that the cubic form $C$ is \textit{strongly equivalent} to a cubic form $C'$ if there is a rational non-singular change of variables, $y'_j := L'_j(\yy)$ such that $C(\xx, \yy) = C'(\xx, \yy')$.  Then we can define $A$ and $M_i$ analogously for $C'$.

\begin{thma}If $11 \leq h \leq n-4$, the cubic form $C$ is strongly equivalent to a form $C'$, for which $M_i=0$ for at least $h - 10$ of the $i$'s.  \end{thma}

\begin{proof}  Note that the set of quadratic forms in 4 variables defined over $k$ is a 10 dimensional vector space under addition over $k$  (since the space of $4 \times 4$ symmetric matrices is).  Consider the quadratic forms $Q_i$ from $\eqref{Qform}$.  We have $h \geq 11$ of them, and we suppose that a maximal linearly independent subset of them has size $t \leq 10$.  It follows that, after relabelling if necessary, we can suppose that $Q_{t+1} , \dots , Q_h \in \txt{span} (Q_1 , \dots , Q_{t})$.  Then considering the first sum of the RHS of $\eqref{Qform}$, we can write  $$\sum_{i=1}^h y_i Q_i (\xx) = \sum_{j=1}^{t} L_j(\yy)Q_j(\xx),$$ for some linear forms $L_j$.  Write $\yy=(\www,\zz)$, with $\www=(y_1,..,y_t)$ and $\zz=(y_{t+1},...,y_h)$.  We observe that $L_j(\www, \textbf{0} )=w_j$ for $1\leq j\leq t$.  Hence if $L_j$ were linearly dependent, then so would be $L_j(\www, \textbf{0})$, which is plainly false.

Let $L_{t+1} , \dots , L_{h}$ be a set of linearly independent $k$-rational forms, none of which lie in the span of $\{L_1, \dots , L_t \}$.   Define $y'_i := L_i(\yy)$ for $i = 1, \dots ,h$.  By construction, this change of variable is non-singular.  Now if we rewrite the form $Q_\yy$ in terms of these new variables, 
\begin{align} Q_\yy(\xx,t) = \sum_{i=1}^h y_i Q'_i (\xx)  + 2 \sum_{j=1}^4 x_j q'_j(\yy) t + c'(\yy) t^2,  \end{align} where $Q'_i = Q_i$ for $i \leq t$, and $Q'_i = 0$ $i >t$.  Hence then the $4 \times 4$ submatrices $M'_i$ given by $Q'_i(\xx)= \, ^T \xx M'_i \xx,$ will all be zero for $i > t$.  There are $h - t \geq h - 10$ such matrices.  \end{proof}

For convenience, we assume that the change of variable from the lemma has already been made, and relabel so that the first $\theta := h - 10$ of the matrices $M_i$ are zero.

Consider the cubic form $C^{\star}(\xx, y_1 , \dots , y_\theta) = C(\xx, y_1 , \dots , y_\theta , 0 , \dots ,0)$, so that only terms dependent on $\xx, y_1 , \dots , y_\theta$ survive.  For our $4 \times 4$ matrices, $M_j = 0$ for $j \in \{ 1, \dots , \theta \}$, so $C^{\star}$ has no quadratic $x$ terms. 

\beq\label{Cstar}C^{\star}(\xx, y_1 , \dots , y_\theta) =  2 \sum_{i=1}^4 x_i q^{\star}_i(y_1 \dots , y_\theta) + c^{\star}(y_1, \dots ,y_\theta). \eeq

Now we recall Lemma $\ref{s-5}$.  Using the argument in the paragraph preceding this lemma, either the assumptions of this lemma are fulfilled for $M_j$, for $j \in \{\theta+1 ,\dots , h \}$, or in the alternative, there is a rational singular point on the cubic hypersurface.  Let $ \{ M'_j \}$ be the set of $3 \times 3$ matrices obtained by deleting the last row and column from each $M_j$.  Then clearly the assumptions of Lemma $\ref{s-5}$ will be inherited by this set.

We split into two cases.

\textbf{Case 1:} Suppose $q^{\star}_i \equiv 0$ for all $i$.  Let's also suppose that $c^{\star} \equiv 0$, so that $C(\xx, y_1 , \dots , y_\theta , 0 , \dots ,0) \equiv 0$.  Then $\{ y_{\theta+1} = \cdots = y_{h} = 0 \}$ is a linear space on $Y$, hence $h \leq 10$, which is a contradiction.  Thus we can assume there is a vector $(b_1 , \dots , b_{\theta}) \in \mathbb{A}_k^\theta$ such that $c^{\star}(b_1 , \dots , b_\theta)\neq 0$.

By Lemma $\ref{s-5}$, for a generic vector $(y_{\theta+1}, \dots , y_h) \in \mathbb{A}_{k}^{10}$, we shall have that $ \sum_{i=\theta+1}^{h} y_i M_i$ has non-zero determinant.  We fix one such generic vector $(b_{\theta+1} , \dots , b_{h})$, and put $M = \sum_{i={\theta+1}}^{h} b_i M_i$.  

In $\eqref{Adef2}$, consider the matrix $$A =  A(P b_1 , \dots , P b_\theta , b_{\theta+1} , \dots , b_{h}).$$  We see that if $P$ is sufficiently large, the \textit{type 1} diagonals dominate the determinant of $A$.  In more detail, note that the bottom right most entry ($A_{5,5}$) will be $P^3 c^{\star}(b_1, \dots , b_\theta) + O(P^2)$, the quadratic $\yy$ terms ($A_{i,5}, A_{5,i}$ , $i \neq 5$) will be at most $O(P)$ since $q^{\star}_i \equiv 0$, and the elements from the leading $4 \times 4$ matrix are independent of $P$.  As in the proof of Lemma $\ref{s-5}$, we put $B = \max_i |b_i|$, $N= \max_j \| M_j \|$.  Then
$$\det (A) = P^3 c^{\star}(b_1, \dots , b_\theta) \det (M) + O_{B,N}(P^2).$$  Taking $P$ sufficiently large, we see the determinant must be non-zero, so we are done.

\textbf{Case 2:}  Suppose that $q^{\star}_4 \not\equiv 0$ (say).  Changing variables amongst $\{ y_1 , \dots , y_\theta \}$ if necessary, we will have a $y_1 ^2$ term present in $q^{\star}_4$ (completing the square for example).  We claim that we can make a linear change variables amongst $\{x_1, \dots , x_4\}$ so that in fact we have a $y_1 ^2$ term present only in $q^{\star}_4$.  To see this, write the quadratic $y$ part of the RHS of $\eqref{Cstar}$ as 
$$2 \sum_{i=1}^4 x_i q^{\star}_i(y_1 \dots , y_\theta) = 2 L(\xx)y^2_1 + \cdots  ,$$ where $L(\xx)$ is non-zero in the $x_4$ term.  Then we have the change of variables $(x'_1,x'_2,x'_3,x'_4) := (x_1,x_2,x_3,L(\xx))$, and this satisfies the claim.  Note that two changes of variables so far do not alter the fact that $M_i=0$ for $i \in \{ 1, \dots , \theta \}$.   

We can assume without loss of generality that these two changes of variable were made at the outset.  Suppose then that $q^{\star}_4(y_1 \dots , y_\theta)$ has the form $a y_1^2 + \cdots$. 

By Lemma $\ref{s-5}$ we know that for the $3 \times 3$ matrices $M'_i$, generically we have $\det \left( \sum_{i={\theta+1}}^{h} y_i M'_i \right) \neq 0$.  We take one such generic vector $(b_{\theta+1} , \dots , b_{h})$, and put $M' = \sum_{i={\theta+1}}^{h} b_i M'_i$.

Now we consider the matrix $$A= A(P , 0 , \dots , 0 , b_{\theta+1} , \dots , b_{h}).$$  This time it is a subset of the \textit{type 2} diagonals which dominates.  The two elements $A_{4,5}$ and $A_{5,4}$ are $a P^2 + O(P)$, the other quadratic $\yy$ terms are $O(P)$, the $A_{5,5}$ entry is $O(P^3)$, and the remaining entries are $O(1)$.  Then if $B = \max_i |b_i|$ and $N= \max_{j=1}^\mu \| M_j \|$, we have $$\det (A) = - a^2 P^4 \det (M') + O_{B,N}(P^3).$$  Taking $P$ sufficiently large we are done.
Thus Lemma $\ref{main3}$ is proved.

\begin{small}

\end{small}

\end{document}